\documentclass[preprint]{article}
\usepackage[left=2cm,right=2cm,
    top=2cm,bottom=2cm]{geometry}
\usepackage{amsmath,amssymb, eucal,graphicx}
\usepackage{amsbsy}
\usepackage{amsthm}
\usepackage[english]{babel}
\usepackage[utf8]{inputenc}
\usepackage{mathtools}
\usepackage{stmaryrd}
\usepackage{bm}
\usepackage{authblk}
\usepackage{cite}

\usepackage[colorlinks=true, allcolors=blue]{hyperref}
\usepackage{booktabs}
\usepackage{graphicx}
\usepackage{subcaption}
\newtheorem{proposition}{Proposition}
\newtheorem{theorem}{Theorem}

\theoremstyle{remark}
\newtheorem{remark}{Remark}

\title{Estimates for the quantized tensor train ranks for the power functions}

\author[1,2]{Sergey A. Matveev\thanks{Corresponding author: matseralex@cs.msu.ru}}
\author[1,2]{Matvey Smirnov}
\affil[1]{Lomonosov MSU, Faculty of Computational Mathematics and Cybernetics, Moscow, Russia, 119991}
\affil[2]{Marchuk Institute of Numerical Mathematics, RAS, Gubkin st. 8, Moscow, Russia, 119333}

\date{}

\begin{document}
\maketitle{}

\begin{abstract}
In this work we provide theoretical estimates for the ranks of the power functions $f(k) = k^{-\alpha}$, $\alpha>1$ in the quantized tensor train (QTT) format for $k = 1, 2, 3, \ldots, 2^{d}$. Such functions and their several generalizations (e.~g. $f(k) = k^{-\alpha} \cdot e^{-\lambda k}, \lambda > 0$)  play an important role in studies of the asymptotic solutions of the aggregation-fragmentation kinetic equations. In order to support the constructed theory we verify the values of QTT-ranks of these functions in practice with the use of the TTSVD procedure and show an agreement between the numerical and analytical results.
\end{abstract}

\section{Introduction}

A lot of models in theoretical physics and mathematics have equilibrium and long-term asymptotic solutions corresponding to universal laws. Knowledge of such basic asymptotics and scalings is extremely useful and might lead to simplified yet correct equations. In particular, such universal observations are possible for the diffusion equation, Helmholtz equation, Fokker-Planck equation, and Smoluchowski aggregation equations (and their generalizations). In this work, we concentrate on the analysis of the power functions 
$$f(k) = k^{-\alpha}$$
arising very often in scaling laws of the aggregation-fragmentation equations \cite{fortin2023stability, leyvraz2003scaling}.  However, numerical verification of such a scaling might require the utilization of a tremendously large number of kinetic equations \cite{matveev2020oscillating}. This fact leads to a certain need for advanced numerical methods and usage of the tensor-based algorithms for efficient storage of the numerical solutions during the calculations.

Originally, the tensor train decomposition has been elaborated for memory-efficient representation of multivariate arrays (a. k. a. tensors) \cite{oseledets2011tensor} with the use of the following idea
\begin{eqnarray}   
\notag
& A(i_1, i_2, \ldots, i_{d-1}, i_d) = \\ \notag
& \\ \label{eq:TT-definition} \\
&  \notag \sum\limits_{\alpha_1=1}^{R_1} \sum\limits_{\alpha_2}^{R_2} \ldots \sum\limits_{\alpha_{d-2}}^{R_{d-2}} \sum\limits_{\alpha_{d-1}}^{R_{d-1}}
    G_1(i_1, \alpha_1)~ G_2(\alpha_1, i_2, \alpha_2) \ldots G_{d-1}(\alpha_{d-2}, i_{d-1}, \alpha_{d-2}) ~ G_d(\alpha_{d-1}, i_d).
\end{eqnarray}

It allows to store the sequence $G_1(i_1, \alpha_1), \ldots, G_d(\alpha_{d-1}, i_d)$ of two- and three-dimensional arrays instead of original tensor $A(i_1, i_2, \ldots, i_{d-1}, i_d)$. This format reduces the memory requirements from $N^d$ memory cells (assume, we use $N \times \ldots \times N$ array) to just $O(d N R^2)$, where $R = \max {R_i}$. The set of minimal possible $R_{i}$ in equation \eqref{eq:TT-definition} is called tensor train ranks of $A$. 
A lot of function-generated tensors (e.g. trigonometric, polynomials, etc) possess exact low-parametric representation with tensor train decomposition.

In addition to a broad family of exactly low-rank tensors there exists a broader family of tensors \cite{tyrtyshnikov2003tensor} with very good low-rank tensor train approximations (e.g. in Frobenius norm) that can be calculated with the use of the close to optimal in terms of accuracy TT-SVD algorithm (and its randomized generalizations) \cite{oseledets2009breaking, sultonov2023low}. The quasioptimality constant for the TT-SVD method is $\sqrt{d-1}$ meaning that
$$\|A - \hat{A} \|_F \leq \sqrt{d-1} \| A - A_{*}\|_F,$$
where $\hat{A}$ is an approximation constructed by the TT-SVD algorithm and $A_{*}$ is the best possible approximation with the same set of the TT-ranks. The total complexity of the TT-SVD method is $O(N^{d+1})$ operations and can be reduced to $O(N^{d})$ with utilization of the randomized approaches without significant loss in terms of the accuracy estimates.

Approximations in the tensor train format also can be found via the much faster but less robust in terms of theoretical accuracy adaptive TT-cross approximation method \cite{oseledets2010tt} requiring evaluation of just $O(d N R^2)$ elements of the target tensor with final cost $O(dNR^3)$ operations. Even though this format has been known for a long time as a matrix product state in quantum physics \cite{white1992density}, only its independent reinvention in numerical analysis as a tensor train format gave birth to these methods leading to constructive approximations of miscellaneous data. Nowadays, there exist many useful and efficient basic algorithms (such as elementwise sum and Hadamard product) \cite{khoromskij2018tensor} for operations with data in tensor-train format as well as complicated optimization procedures such as alternating minimal energy methods for linear systems in higher dimensions \cite{dolgov2014alternating}  or global optimization procedures \cite{zheltkov2020global}.

The tensor train has shown its surprisingly great performance for data compression and scientific computing with the introduction of virtual dimensions into even one-dimensional arrays \cite{oseledets2009approximation}. In simple words, this idea allows to reshape a vector with $N=2^d$  entries as d-dimensional tensor $\underbrace{2 \times \ldots  \times 2}_{\text{d times}}$ and its further decomposition into the tensor train format. If the corresponding tensor train ranks are \textit{sufficiently small} we observe a tremendous compression of initial data to $O(d \cdot 2 \cdot R^2)$ memory cells meaning the logarithmic requirements $O(\log N \cdot R^2)$ for storage of either initial or approximated data.  Such a trick is known as the quantized tensor train (QTT) format and becomes useful for a really wide number of fundamental and applied problems.

However, accurate analytical estimates for approximate QTT ranks of many functions are still under the interest of researchers. One of the most recent successes seems to be the result on QTT-ranks of polynomial functions \cite{vysotsky2021tt} and inverses of circulant matrices \cite{vysotsky2023tensor}.  Another interesting yet experimental observation seems to be the application of the quantized tensor train format for storage of the numerical solution during the iterations of Newton method \cite{timokhin2020tensorisation} solving the aggregation-fragmentation equations with $d$ up to $40$. In this set of experiments, asymptotic solutions are close to the power functions and the QTT-ranks of their numerical approximations do not exceed 50 for approximation error $10^{-14}$. In this work we provide both theoretical and experimental research of the QTT-ranks for the decaying power functions $f(k) = k^{-\alpha}$, $\alpha > 1$.

\section{Estimates for QTT-ranks}
For this section we fix a real sequence $f = \{f(k)\}_{k = 1}^\infty$. Given $d \in \mathbb N$ we denote by $Q(d,f)$ the tensor with size $2 \times \dots \times 2 = 2^d$ obtained by natural reshaping of the vector $(f(1), \dots, f(2^d))$. The $s$-unfolding matrix (that is, the matrix obtained by merging first $s$ indices and last $d-s$ indices of $Q(d,f)$ into multi-indices) of $Q(d, f)$ has the size $2^s \times 2^{d-s}$ and is given by
\begin{equation}
    A(s,d,f) = \begin{pmatrix}
        f(1) & f(2) & \dots & f(2^{d-s}) \\
        f(2^{d-s} + 1) & f(2^{d-s} + 2) & \dots & f(2^{d-s+1}) \\
        \vdots & \vdots & \ddots & \vdots \\
        f((2^{s}-1)2^{d-s} + 1) & f((2^s - 1)2^{d-s} + 2) & \dots & f(2^d)
    \end{pmatrix}.
\end{equation}
That is, an element of this matrix has the form $$A(s,d,f)_{ij} = f((i-1)2^{d-s} + j).$$
The ranks of such unfolding matrices are exactly the minimal TT-ranks of the corresponding tensor as well as their numerical approximations lead to the numerical construction of the tensor train via the TTSVD algorithm. 
The estimates on singular values of matrices $A(s,d,f)$ can be derived by considering the following Hankel matrix

\begin{equation}
    H(d,f) = \
    \begin{pmatrix}
        f(1) & f(2) & \dots & f(2^{d-1}) \\
        f(2) & f(3) & \dots & f(2^{d-1} + 1) \\
        \vdots & \dots & \ddots & \vdots \\
        f(2^{d-1}) & f(2^d + 1) & \dots & f(2^d)
    \end{pmatrix}.
\end{equation}
Clearly, $A(s,d,f)$ is a submatrix in $H(d,f)$ for $s = 1,\dots,d-1$.

Our analysis is based on the estimates of singular values of positive semi-definite Hankel matrices (see~\cite{BeckTown, teeHankel}). 
Recall that a sequence $f$ gives rise to the Hamburger moment problem. This problem asks whether there exists a positive Borel measure on $\mathbb R$ such that
\begin{equation}\label{eqHamburger}
f(k) = \int_{\mathbb R} t^{k-1} d\mu(t),\;k \in \mathbb N.
\end{equation}
It is clear that the Hamburger moment problem for the sequence $f(k) = k^{-\alpha}$, $\alpha > 0$ admits a solution, since~\eqref{eqHamburger} is satisfied with 
$$d\mu(t) = \chi_{[0,1]} (-\ln t)^{\alpha - 1}dt / \Gamma(\alpha),$$
where $\chi_A$ denotes the characteristic function of a set $A \subset \mathbb R$.

As far as QTT-decomposition of $e^{-\beta k}$ has all ranks equal to 1, the QTT-ranks of $f(k) \cdot e^{-\beta k}$ are less or equal to the QTT-ranks of $f(k)$ because the corresponding decomposition of $f(k)$ can be multiplied elementwise by the QTT-representation of exponential function without change of ranks.
\begin{proposition}\label{prBasicPropertiesOfH}
    ~\\\begin{enumerate}
        \item\label{prBasicI} $H(d,f)$ is positive semi-definite for all $d \in \mathbb N$ if and only if the Hamburger moment problem for $f$ is solvable.
        \item\label{prBasicII} Assume that there exists an essentially bounded periodic function $g(t)$ on $\mathbb R$ such that 
        $$f(k) = \frac{1}{2 \pi}\int_0^{2\pi} g(t)e^{-ikt}dt, \quad k \in \mathbb N.$$
        Then
        $$\sup_{d \in \mathbb N} \|H(d,f)\|_2 =  \lim_{d \rightarrow \infty} \|H(d,f)\|_2 \le \mathrm{esssup}_{t \in \mathbb R} |g(t)|.$$
    \end{enumerate}
\end{proposition}
\begin{proof}
    The proof of~\eqref{prBasicI} can be found in~\cite[Theorem~X.4]{ReedSimon}. The statement~\eqref{prBasicII} is a particular case of~\cite[Theorem~I]{Nehari}.
\end{proof}

\begin{proposition}\label{prEstimatesOnH}
    ~\\Assume that the Hamburger moment problem for $f$ is solvable and let $$q(d) = \exp\left( \frac{\pi^2}{4((d + 2)\ln2 - \ln\pi)}\right).$$ Then the following statements hold.
    \begin{enumerate}
        \item\label{prEstI} The singular values of $H(d,f)$ satisfy inequalities
        $$
            \sigma_{2k + 1} (H(d,f)) \le 16 q(d)^{-2k+2}\|H(d,f)\|_2.
        $$
        \item\label{prEstII} There exists a constant $C_1 > 0$ independent of $f$ and $d$ such that the distance from $H(d,f)$ to matrices of rank $r$ in Frobenius norm is estimated from above by
        $$
            C_1 \|H(d,f)\|_2 dq(d)^{-r}.
        $$
        \item\label{prEstIII} There exists a constant $C_2 > 0$ independent of $f$ such that for $\varepsilon \in (0,1)$ and $d \in \mathbb N$ there is a matrix $H$ such that $\|H(d,f) - H\|_F \le \varepsilon$ and $$\mathrm{rank} H \le C_2 d (\ln (d) + \ln(1/\varepsilon) + \ln(\|H(d,f)\|_2) +1).$$ 
    \end{enumerate}
\end{proposition}
\begin{proof}
    The statement~\eqref{prEstI} is a direct corollary of \cite[Corollary~5.5]{BeckTown}, since $H(d,f)$ is a positive semi-definite Hankel matrix. To prove~\eqref{prEstII} note that the distance in Frobenius norm from $H(d,f)$ to matrices of rank $r$ is equal to
    $$
        \sqrt{\sum_{k = r+1}^\infty \sigma_k (H(d,f))^2}.
    $$
    Now the proof of~\eqref{prEstII} can be finished using estimates from~\eqref{prEstI} and summing the corresponding geometric series. Finally, statement~\eqref{prEstIII} follows from solving the following inequality for $r$
    $$
        C_1 \|H(d,f)\|_2 dq(d)^{-r} \le \varepsilon.
    $$
\end{proof} 
\begin{theorem}\label{th1}
    Assume that the Hamburger moment problem for $f$ is solvable and that there exists a function $g(t)$ satisfying the conditions of Proposition~\ref{prBasicPropertiesOfH}~\eqref{prBasicII}. Let $M = \mathrm{esssup}_{t \in \mathbb R} |g(t)|$. There exists a constant $C > 0$ independent of $f$ such that for all $\varepsilon \in (0,1)$ and all $d \in \mathbb N$ there exists a tensor $T$ of the size $2 \times \dots \times 2 = 2^d$ with tensor train ranks not exceeding $C d (\ln(d) + \ln(1/\varepsilon) + \ln(M) + 1)$ such that $\|Q(d, f) - T\|_F \le \varepsilon$.
\end{theorem}
\begin{proof}
    We consider the approximation of the tensor $Q(d, f)$ obtained by the TT-SVD procedure \cite[Section~2]{oseledets2011tensor} with low-rank approximation of each unfolding matrix with error not exceeding $\varepsilon / \sqrt{d-1}$ in Frobenius norm. By~\cite[Theorem~2.2]{oseledets2011tensor} the obtained tensor $T$ satisfies 
    $$\|Q(f) - T\|_F \le \varepsilon$$
    and its tensor train ranks do not exceed ranks of the approximations of corresponding unfolding matrices. 
    
    Since all unfolding matrices of $Q(d, f)$ are submatrices in $H(d,f)$ by Proposition~\ref{prEstimatesOnH}~\eqref{prEstIII} we obtain that it is possible to choose approximations with ranks not exceeding 
    $$ C_2 d (\ln (d) + \ln(\sqrt{d-1}/\varepsilon) + \ln(\|H(d,f)\|_2) + 1).
    $$
    Now to finish the proof it remains to apply the inequality $\|H(d,f)\|_2 \le M$ from Proposition~\ref{prBasicPropertiesOfH}~\eqref{prBasicII}.
\end{proof}

\begin{figure}[ht!]
    \centering
    \includegraphics[scale=0.6]{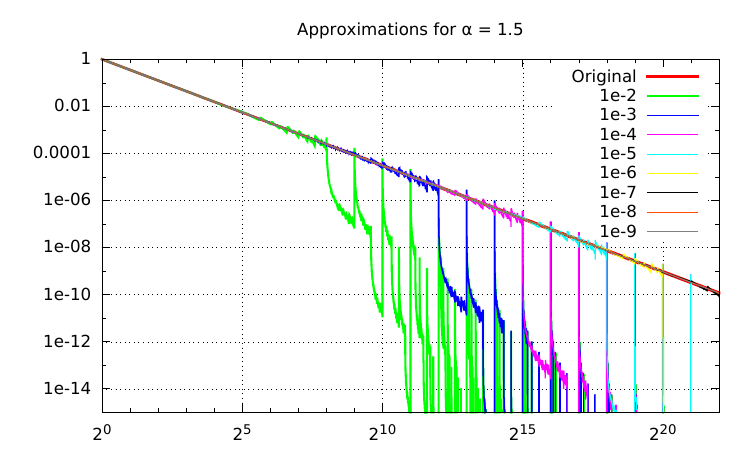}~
    \includegraphics[scale=0.6]{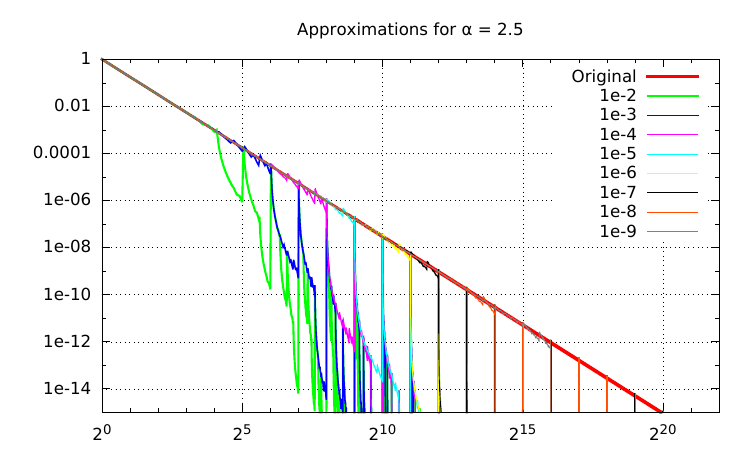}
    \caption{QTT approximations with multiple tolerance levels for $d=22$ and $\alpha=1.5, 2.5$. Approximation with TTSVD accuracy parameter $\varepsilon=10^{-9}$ requires just maximal QTT-rank $R=7$, each increase of this rank by 1 allows to increase the accuracy of the approximation by dozens of times.}
    \label{fig:TTSVD}
\end{figure}

\begin{remark}
    Note that if $f$ is summable, then a function $g$ satisfying requirements of Proposition~\ref{prBasicPropertiesOfH}~\eqref{prBasicII} exists. Indeed, in this case we can set $g(t) = \sum\limits_{k = 1}^\infty f(k) e^{ikt}$. This function is obviously bounded and $|g(t)| \le \sum\limits_{k = 1}^\infty |f(k)|$. So, assuming in addition that the Hamburger moment problem for $f$ is solvable, it is possible to apply Theorem~\ref{th1} with $M = \sum\limits_{k = 1}^\infty |f(k)|$.
\end{remark}
\begin{remark}
    The sequence $f(k) = 1/k$ also satisfies the requirements of Theorem~\ref{th1} even though it is not summable. Indeed, consider $g(t) = \sum\limits_{k = 1}^\infty (e^{ikt} - e^{-ikt})/ k = 2i \sum\limits_{k = 1}^\infty \sin(kt)/k$. It is known that $g(t)$ is essentially bounded and, therefore, it satisfies the conditions of Proposition~\ref{prBasicPropertiesOfH}~\eqref{prBasicII}.
\end{remark}

\section{Numerical verification}

To verify our theoretical investigations we apply the TTSVD algorithm with multiple tolerance values (from $\varepsilon=10^{-2}$ to $10^{-9}$), $d$ and several values of $\alpha$. In our experiments, we observe really low numerical QTT-ranks of the constructed approximations in agreement with asymptotics. We demonstrate the results of TT approximations in Figure \ref{fig:TTSVD} for the fixed $d=22$. 

\begin{table}
    \centering
    \begin{tabular}{p{1.5 cm}|c|c|c|c|c|c|c|c}
    \hline
         Accuracy & $10^{-2}$ & $10^{-3}$ & $10^{-4}$ & $10^{-5}$
         & $10^{-6}$ & $10^{-7}$ & $10^{-8}$ & $10^{-9}$\\
         \hline
         $\bf{\alpha = 2.5}$ & $R ~~ \hat{r}$  & $R ~~ \hat{r}$ &  $R ~~ \hat{r}$ &  $R ~~ \hat{r}$ &  $R ~~ \hat{r}$ &  $R ~~ \hat{r}$ &  $R ~~ \hat{r}$ &  $R ~~ \hat{r}$ \\
         \hline
         $d=30$ &  2 ~~ 1.07 &  3 ~~ 1.20 &  3 ~~ 1.33 &  4 ~~ 1.50 &  4 ~~ 1.70 &  5 ~~ 1.97 &  6 ~~ 2.17 &  6 ~~ 2.50 \\
         \hline
         $d=28$ &  2 ~~ 1.07 &  3 ~~ 1.21 &  3 ~~ 1.36 &  4 ~~ 1.54 &  4 ~~ 1.75 &  5 ~~ 2.04 &  6 ~~ 2.25 &  6 ~~ 2.61  \\
         \hline
         $d=26$ &  2 ~~ 1.08 &  3 ~~ 1.23 &  3 ~~ 1.35 &  4 ~~ 1.58 &  4 ~~ 1.81 &  5 ~~ 2.08 &  6 ~~ 2.35 &  6 ~~ 2.69  \\
         \hline
         $d=24$ &  2 ~~ 1.08 &  3 ~~ 1.21 &  3 ~~ 1.38 &  4 ~~ 1.62 &  4 ~~ 1.83 &  5 ~~ 2.12 &  6 ~~ 2.46 &  6 ~~ 2.83 \\
         \hline
         $d=22$ & 2 ~~ 1.09 &  3 ~~ 1.23 &  3 ~~ 1.41 &  4 ~~ 1.68 &  4 ~~ 1.91 &  5 ~~ 2.18 &  6 ~~ 2.59 &  6 ~~ 3.00  \\
         \hline
         $d=20$ &  2 ~~ 1.10 &  3 ~~ 1.25 &  3 ~~ 1.45 &  4 ~~ 1.75 &  4 ~~ 2.00 &  5 ~~ 2.30 &  6 ~~ 2.75 &  6 ~~ 3.20  \\
        \hline
        $d=18$ &  2 ~~ 1.11 &  3 ~~ 1.28 &  3 ~~ 1.44 &  4 ~~ 1.78 &  4 ~~ 2.11 &  5 ~~ 2.44 &  6 ~~ 2.94 &  6 ~~ 3.33 \\
        \hline        
        $d=16$ &  2 ~~ 1.12 &  3 ~~ 1.31 &  3 ~~ 1.50 &  4 ~~ 1.88 &  4 ~~ 2.25 &  5 ~~ 2.62 &  5 ~~ 3.12 &  6 ~~ 3.62  \\
        \hline 
        $d=14$ &  2 ~~ 1.14 &  3 ~~ 1.36 &  3 ~~ 1.57 &  4 ~~ 2.00 &  4 ~~ 2.29 &  5 ~~ 2.86 &  5 ~~ 3.29 &  6 ~~ 3.93  \\
        \hline
        $d=12$ &  2 ~~ 1.17 &  3 ~~ 1.42 &  3 ~~ 1.67 &  4 ~~ 2.17 &  4 ~~ 2.50 &  5 ~~ 3.17 &  5 ~~ 3.50 &  6 ~~ 3.92 \\
        \hline    
         \hline
        $\bf{\alpha = 1.5}$ & ~  & ~ &  ~ &  ~ &  ~ &  ~ &  ~ &  ~\\
        \hline
         $d=30$ &  3 ~~ 1.30 &  3 ~~ 1.53 &  4 ~~ 1.90 &  5 ~~ 2.30 &  5 ~~ 2.77 &  6 ~~ 3.27 &  6 ~~ 3.80 &  7 ~~ 4.47  \\
         \hline
         $d=28$ &  3 ~~ 1.32 &  3 ~~ 1.57 &  4 ~~ 1.96 &  5 ~~ 2.39 &  5 ~~ 2.86 &  6 ~~ 3.43 &  6 ~~ 4.00 &  7 ~~ 4.61  \\
         \hline
         $d=26$ &  3 ~~ 1.35 &  3 ~~ 1.62 &  4 ~~ 2.04 &  5 ~~ 2.50 &  5 ~~ 2.96 &  6 ~~ 3.62 &  6 ~~ 4.15 &  7 ~~ 4.69  \\
         \hline
         $d=24$ &   3 ~~ 1.33 &  3 ~~ 1.67 &  4 ~~ 2.12 &  5 ~~ 2.54 &  5 ~~ 3.12 &  6 ~~ 3.71 &  6 ~~ 4.29 &  7 ~~ 4.83 \\
         \hline
         $d=22$ &  3 ~~ 1.36 &  3 ~~ 1.73 &  4 ~~ 2.23 &  5 ~~ 2.68 &  5 ~~ 3.23 &  6 ~~ 3.91 &  6 ~~ 4.41 &  7 ~~ 4.95 \\
         \hline
         $d=20$ &  3 ~~ 1.40 &  3 ~~ 1.75 &  4 ~~ 2.30 &  4 ~~ 2.75 &  5 ~~ 3.45 &  6 ~~ 4.00 &  6 ~~ 4.40 &  7 ~~ 5.00 \\
        \hline
        $d=18$ &   3 ~~ 1.44 &  3 ~~ 1.83 &  4 ~~ 2.39 &  4 ~~ 2.94 &  5 ~~ 3.56 &  6 ~~ 4.06 &  6 ~~ 4.50 &  7 ~~ 5.00 \\
        \hline        
        $d=16$ &  3 ~~ 1.50 &  3 ~~ 1.94 &  4 ~~ 2.56 &  4 ~~ 3.06 &  5 ~~ 3.62 &  6 ~~ 4.12 &  6 ~~ 4.50 &  7 ~~ 5.00\\
        \hline 
        $d=14$ &  3 ~~ 1.50 &  3 ~~ 2.00 &  4 ~~ 2.64 &  4 ~~ 3.14 &  5 ~~ 3.71 &  6 ~~ 4.21 &  6 ~~ 4.50 &  7 ~~ 4.93 \\
        \hline
        $d=12$ & 2 ~~ 1.50 &  3 ~~ 2.08 &  4 ~~ 2.83 &  4 ~~ 3.17 &  5 ~~ 3.67 &  6 ~~ 4.08 &  6 ~~ 4.33 &  7 ~~ 4.75\\
        \hline     
    \end{tabular}
    \caption{Maximal QTT-ranks $R$ for multiple values $d$ and varying accuracy $\varepsilon$ for the TTSVD algorithm. We fix $\alpha = 2.5 $ and $\alpha = 1.5 $. }
    \label{tab:QTT_ranks}
\end{table}

We also demonstrate the logarithmic increase of maximal QTT ranks in Table \ref{tab:QTT_ranks}. Even though the whole set of ranks changes slightly to a variation of $d$ there is almost no change in $R$. We indicate this observation showing the average ranks $\hat{r}=\frac{1}{d}\sum\limits_{i=1}^d R_i$ in Table \ref{tab:QTT_ranks}. It agrees with the Theorem from the previous section quite well: one needs to have a really big change of $d$ to obtain logarithmic growth of $R$ (but even $d=30$ corresponds to vectors with more than a billion of entries).  At the same time, we see a gradual logarithmic increase of $R$ with respect to accuracy as predicted by theory. These results agree quite well with numerical observations from \cite{timokhin2020tensorisation} for more complicated calculations but without theoretical estimates.

We see that $\hat{r}$ decreases with respect to increasing $d$ for the fixed levels of the approximation accuracy. Such behaviour of $\hat{r}$ is rather natural. As soon as Frobenius norm of the ``tail'' satisfies 
$$\sqrt{\sum_{k=N}^{\infty} k^{-2\alpha}} <\varepsilon$$
the QTT-ranks for the corresponding virtual indices $k > N$ stay constant and equal to one for any $d$ without significant influence on the total accuracy of the approximation.

\section{Conclusion}

We have studied the ranks of the approximate quantized tensor train decompositions for the power functions $f(k) = k^{-\alpha}$ and prove that $R$ should grow logarithmically with respect to the accuracy of the approximation in Frobenius norm as well as to the size of the vector $N = 2^d$. We verify our theory numerically with the use of the TTSVD algorithm and obtain its good agreement with numerical results. These results allow us to get analytical justification for the efficiency of recent numerical calculations for aggregation-fragmentation equations representing the solution in QTT-format \cite{timokhin2020tensorisation}.

\section*{FUNDING}

This work was supported by the Russian Science Foundation (project no.  \href{https://www.rscf.ru/project/21-71-10072/}{21-71-10072}).

\end{document}